\documentclass[reqno,11pt]{amsart}

\usepackage[normalem]{ulem}
\usepackage{mathrsfs,amsmath,amssymb,amsthm,bbm}
\usepackage{hyperref}

\usepackage{setspace}

\newtheorem{lem}{Lemma}[section]
\newtheorem{thrm}[lem]{Theorem}
\newtheorem{conj}[lem]{Conjecture}
\newtheorem{prop}[lem]{Proposition}
\newtheorem{cor}[lem]{Corollary}
\theoremstyle{definition}

\theoremstyle{remark}
\newtheorem*{rem}{Remark}

\renewcommand{\Re}{\ensuremath{\operatorname{Re}}}

\renewcommand{\epsilon}{\varepsilon}

\newcommand{\mr}{\mathrm}
\newcommand{\mf}{\mathfrak}

\newcommand{\iu}{{i\mkern1mu}}

\newcommand{\R}{\ensuremath{\mathbb{R}}}
\newcommand{\C}{\ensuremath{\mathbb{C}}}

\newcommand{\Schwartz}{\ensuremath{\mathscr{S}}}

\DeclareMathOperator{\sgn}{sgn}
\newcommand{\<}{\ensuremath{\langle}}
\renewcommand{\>}{\ensuremath{\rangle}}

\newcommand{\p}{\ensuremath{\partial}}

\newcommand{\bbo}{\ensuremath{\mathbbm 1}}

\newcommand{\bpf}{\begin{proof}}
\newcommand{\epf}{\end{proof}}

\DeclareMathOperator{\tr}{tr}
\newcommand{\I}{\mf I}
\newcommand{\op}{\mr{op}}
\newcommand{\DNLS}{\mr{DNLS}}

\newcommand{\Bd}{B_\delta}
\newcommand{\BdS}{\Bd\cap \Schwartz}

\newcommand{\vk}{\varkappa}
\newcommand{\pd}{\,{\rm{d}}}
\newcommand{\dx}{\,{\rm{d}}x}
\newcommand{\dy}{\,{\rm{d}}y}
\newcommand{\norm}[1]{{\left\vert\kern-0.25ex\left\vert\kern-0.25ex\left\vert #1 \right\vert\kern-0.25ex\right\vert\kern-0.25ex\right\vert}}

\newcommand{\qtq}[1]{\quad\text{#1}\quad}
\newcommand{\vr}{\gamma}

\newcommand{\sbrack}[1]{^{[#1]}}

\title[microscopic conservation laws for DNLS]{microscopic  conservation laws for the derivative Nonlinear Schr\"{o}dinger equation}

\numberwithin{equation}{section}
\allowdisplaybreaks

\begin{document}
	\onehalfspacing

	\author[]{Xingdong Tang}
	\address{\hskip-1.15em Xingdong Tang 
		\hfill\newline School of Mathematics and Statistics, \hfill\newline Nanjing Univeristy of Information Science and Technology, 
		\hfill\newline Nanjing, 210044,  People's Republic of China.}
	\email{txd@nuist.edu.cn}
	
	\author[]{Guixiang Xu}
	\address{\hskip-1.15em Guixiang Xu
		\hfill\newline Laboratory of Mathematics and Complex Systems,
		\hfill\newline Ministry of Education,
		 \hfill\newline School of Mathematical Sciences,
		\hfill\newline Beijing Normal University,
		\hfill\newline Beijing, 100875, People's Republic of China.}
	\email{guixiang@bnu.edu.cn}

	

\begin{abstract}
	Compared with macroscopic conservation law for the solution of the derivative nonlinear Schr\"odingger equation (DNLS) with small mass in \cite{KlausS:DNLS}, we show the corresponding microscopic conservation laws for the Schwartz solutions of  DNLS with small mass. The new ingredient is to make use of the logarithmic perturbation determinant  introduced  in \cite{Rybkin:KdV:Cons Law, Simon:Trace} to show one-parameter family of microscopic conservation laws of the $A(\kappa)$ flow and the DNLS flow, which is motivated by \cite{HKV:NLS,KV:KdV:AnnMath,KVZ:KdV:GAFA}.  
\end{abstract}

\maketitle

%

\section{Introduction}
We consider the derivative nonlinear Schr\"odinger equation (DNLS)
\begin{equation}\label{DNLS}
\iu \partial_t q + \partial^2_x q + \iu  \partial_x(|q|^2q)  =0,
\end{equation}
where  \(q\colon \R\times \R\rightarrow \C\). \eqref{DNLS}  is $L^2$-critical since  the dilation
\begin{align}\label{scaling}
q(t,x)\mapsto q_{\lambda}(t,x)=\lambda^{1/2}q(\lambda^2t, \lambda x)
\end{align}
leaves both \eqref{DNLS} and the  $L^2$ norm invariant.   The derivative nonlinear Schr\"odinger equation appears in plasma physics \cite{MOMT-PHY, M-PHY,
	SuSu-book}, and references therein.

Local well-posedness result for \eqref{DNLS} in the energy space was
worked out by N. Hayashi and T. Ozawa \cite{HaOz-94, Oz-96}. They
combined the fixed point argument with the $L^4_tW^{1, \infty}_x$
estimate to construct local-in-time solution with arbitrary data in
energy space. For other results, we can refer to \cite{Ha-93,
	HaOz-92}. Since \eqref{DNLS} is energy subcritical case, the maximal time interval of existence only depends  on $H^1$ norm of
initial data. Later, local well-posedness result for \eqref{DNLS}  in $H^s, s\geq 1/2$ is due to
H. Takaoka \cite{Ta-99} by Bourgain's Fourier restriction method.
The sharpness is shown in \cite{Ta-01} in the sense that nonlinear
evolution $u(0)\mapsto u(t)$ fails to be $C^3$ or even uniformly
$C^0$ in this topology, even when $t$ is arbitrarily close to zero
and $H^s$ norm of the data is small (see also Biagioni-Linares
\cite{BiLi-01-Illposed-DNLS-BO}).

Global well-posedness is shown for \eqref{DNLS}
in the energy space in \cite{Oz-96},  under the smallness condition
\begin{align}\label{Cond:smalldata}
\|u_0\|^2_{L^2} < 2 \pi,
\end{align}
the argument is based on the sharp Gagliardo-Nirenberg inequality
and the energy method (conservation of mass and energy). This result is
improved by H. Takaoka \cite{Ta-01} by Bourgain's
restriction method, who proved global
well-posedness in $H^s$ for $s>32/33$ under the condition
\eqref{Cond:smalldata}.  In \cite{CKSTT-01, CKSTT-02}, I-team make use of almost conservation law \cite{Tao:book:Nonlinear Dispersive Equations} to
show global well-posedness in $H^s, s>1/2$ under
\eqref{Cond:smalldata}. Miao, Wu and Xu \cite{MiaoWX-2011} combine almost conservation law and the refined resonant
decomposition technique to  obtain the
 global well-posedness in $H^{1/2}$ under
\eqref{Cond:smalldata}.  Later, Wu use the generalized Gagliardo-Nirenberg inequality to improve the global well-posedness  of \eqref{DNLS} in the energy space under the condition
\begin{align}\label{Cond:small:4pi}
\|u_0\|^2_{L^2} < 4 \pi
\end{align}
in \cite{Wu-DNLS}, where $4\pi$ is the mass of the solitary waves with critical parameters of \eqref{DNLS}.  Miao, Tang and Xu use the structure analysis and classical variational argument to show the existence of solitary waves with two parameters and  improve the global result of \eqref{DNLS} in the energy space in \cite{MTX:DNLS:Exist}, and further use perturbation argument, modulation analysis and Lyapunov stability to show the orbital stability of weak interaction multi-soliton solution with subcritical parameters  in the energy space in \cite{MTX:DNLS:stab}.
 We can also refer  to \cite{ColOhta:DNLS:stab, GNW:gDNLS:instab, LeW:gDNLS:stab, MTX:gDNLS:instab, TX:gDNLS:stab} for  the stability analysis of the solitary waves of the (generalized)  derivative nonlinear Schr\"odinger equation in the energy space and to \cite{GW:DNLS:GWP} for lower regularity result of \eqref{DNLS} by almost conservation law in \cite{Tao:book:Nonlinear Dispersive Equations}. 
 
Since \eqref{DNLS} is an integrable system in \cite{AbCl:book, KaupN:DNLS},  there are lots of global well-posedness of \eqref{DNLS} with mass restriction in the  weighted Sobolev spapce based on the inverse scattering method, please refer to \cite{JLPS:DNLS:APDE, JLPS:DNLS:QJPAM, JLPS:DNLS:CPDE, PelinSS:DNLS:DPDE, PelinSS:DNLS:IMRN}and reference therein.

The conjecture about \eqref{DNLS} is the following.

\begin{conj}
	Let \(s>0\). \eqref{DNLS}  is globally well-posed for all initial data in \(H^s(\R)\)  in the sense that the solution map \(\Phi\) extends uniquely from Schwartz space to a jointly continuous map \(\Phi\colon \R\times H^s(\R)\rightarrow H^s(\R)\).
\end{conj}

According to the above well known result, we need loosen the continuous dependence of the solution on initial data to consider the solution of \eqref{DNLS} in $H^s(\R)$ with $s\in (0,1/2)$.  The basic question is that how to control  the uniform estimates of the solution of \eqref{DNLS}. 

Motivated by Killip-Visan-Zhang's argument in \cite{KVZ:KdV:GAFA}, Klaus and Schippa combine the integrability of \eqref{DNLS} with the series expansion of the perturbation determinant \cite{Rybkin:KdV:Cons Law, Simon:Trace} to obtain  macroscopic conservation law for the Schwartz solution of \eqref{DNLS} with small mass in \cite{KlausS:DNLS}, In this paper, we will   show the corresponding microscopic form and obtain one-parameter family of microscopic conservation laws for the $A(\kappa)$  flows (see \eqref{A-def}) and the  DNLS flow by Harrop-Griffiths-Killip-Visan's argument in \cite{HKV:NLS, KV:KdV:AnnMath}. Compared with macroscopic form, microscopic conservation law with coercivity helps to show the local smoothing effect for \eqref{DNLS} in $H^s(\R)$ and can be further applied into global wellposedness analysis.  We can refer to  \cite{HKV:NLS,KV:KdV:AnnMath, Tal:BO} and reference therein.  
 
 We now  recall some Hamiltonian mechanics background. \eqref{DNLS} is Hamiltonian equation with respect to the following Poisson structure: 
\begin{equation}\label{PoissonBracket}
	\{F,G\} :
	=  \int  \tfrac{\delta F}{\delta r}\partial_x \bigl(\tfrac{\delta G}{\delta q}\bigr) + \tfrac{\delta F}{\delta q}\partial_x \bigl(\tfrac{\delta G}{\delta r}\bigr) \dx,
\end{equation} 
where the operators $\tfrac{\delta }{\delta q}$ and $\tfrac{\delta }{\delta r}$ denote the functional Fr\'echet derivatives. Any Hamiltonian $H(q, r, t)$ generates a flow via the equation
\begin{align}\label{HFlow}
\partial_t  \begin{bmatrix} q \\ r \end{bmatrix}
=   \begin{bmatrix} 0 & 1  \\ 1 & 0 \end{bmatrix} \partial_x 
\begin{bmatrix} \frac{\delta H}{\delta q} \\ \frac{\delta H}{\delta r} \end{bmatrix}. 
\end{align}
Correspondingly, \eqref{DNLS} is the Hamiltonian flow associated to
\begin{equation}\label{HDNLS}
H_{\DNLS} := \int _{\R}-\iu qr' + \frac12 q^2r^2\dx,
\end{equation}
where  $r=-\bar{q}$. Two other important Hamiltonian quantities for \eqref{DNLS} are
\begin{equation}\label{Mass Energy}
M:=   \int qr\dx, \quad E_{\DNLS} :=\int q' r' - \frac32 \iu q^2 rr' + \frac12 q^3 r^3 \dx.
\end{equation}
Conservations of $M$, $H_{\DNLS}$  and $E_{\DNLS}$ is due to  gauge, space translation and time translation invariance of \eqref{DNLS}, the commutativity of $H_{\DNLS}$ and $E_{\DNLS}$  is based on the fact that they are completely integrable, which means the existence of an infinite family of commuting flows for \eqref{DNLS}.  We can refer to \cite{AbCl:book} \cite{KlausS:DNLS}  for more details.  Based on the recent breakthrough in KdV, mKdV and NLS by B. Harrop-Griffiths, R. Killip and M. Visan  in \cite{HKV:NLS, KV:KdV:AnnMath}, the commuting flow approximation  will be the robust  method for showing global well-posedness theory of \eqref{DNLS} in the lower regularity space.  We can also refer to \cite{KTataru:NLS:DMJ, NRTataru:DS:InventM}  and reference thereein.

Let us write the Lax operator related to \eqref{DNLS} and its unperturbed one
\begin{equation}\label{Intro KN L}
L(\vk) := \begin{bmatrix}\p+\iu\vk^2 & -\vk q\\ -\vk r&\p-\iu\vk^2\end{bmatrix} \text{~~and~~} L_0(\vk) := \begin{bmatrix}\p+\iu\vk^2  & 0\\0& \p-\iu\vk^2 \end{bmatrix}.
\end{equation}
By simple calculations, we know that 
\begin{equation*}
R_0(\kappa) := L_0(\kappa)^{-1} = \begin{bmatrix}(\p+\iu\kappa^2)^{-1} & 0\\0&(\p-\iu\kappa^2)^{-1}\end{bmatrix}
\end{equation*}
admits the integral kernel
\begin{equation*}\label{G_0}
G_0(x,y;\kappa) = e^{-\iu\kappa^2|x - y|}\begin{bmatrix}\bbo_{\{y<x\}}&0\\0&-\bbo_{\{x<y\}}\end{bmatrix} \quad\text{for $\iu\kappa^2 >0$}.
\end{equation*}
For $\iu\kappa^2<0$, we may use ${G}_0(x,y;\kappa)=-G_0(y,x;-\bar{\kappa})$. The resolvent operator $R(\kappa):=L(\kappa)^{-1}$ for Schwartz function $q$ with small mass also has the  integral kernel $G(x,y;\kappa)$ (See Proposition \ref{P:R})
\begin{equation*} 
\begin{bmatrix}
G_{11}(x,y,\kappa) & G_{12}(x,y,\kappa) \\
G_{21}(x,y,\kappa) & G_{22}(x,y,\kappa)
\end{bmatrix}.
\end{equation*}
Let us define three key functionals as follows.
\begin{align*}
\vr(x;\kappa) &:=\sgn(\iu\kappa^2) \bigl[ G_{11}(x,x;\kappa) - G_{22}(x,x;\kappa) \bigr]  - 1, \\
g_{12}(x;\kappa) &:=  \sgn(\iu\kappa^2) G_{12}(x,x;\kappa), \\
g_{21}(x;\kappa) &:=  \sgn(\iu\kappa^2) G_{21}(x,x;\kappa), 
\end{align*}
 
 With these preparations,  the main result in this paper is
\begin{thrm}
	Let $s\in (0,1/2)$, $\iu\vk^2\in\R\setminus(-1,1)$,  $q(0)\in H^s (\R)\cap \Schwartz$ with small mass. Suppose that $q$ is a  solution to \eqref{DNLS}, then we have
	 $$\partial_t \rho(\vk) +  \partial_x j_{\DNLS}(\vk) =0, $$ where
the density $	\rho$  and the flux $	j_{\DNLS}$ are defined as follows	\begin{align*}
	\rho(\vk): =\, & -\vk\frac{qg_{21}(\vk)+rg_{12}(\vk)}{2+\vr(\vk)},\\
	j_{\DNLS}(\vk):
	= \, &
	i \vk \frac{q'g_{21}(\vk) - r' g_{12}(\vk)}{2+\vr(\vk)}
	- i\vk^2 q r
	- 2\vk^2 \rho(\vk)
	- qr\rho(\vk). 
	\end{align*}
	\end{thrm}

\begin{rem} We give some remarks as follows. 
	\begin{enumerate}
		 \item In this paper, we consider microscopic conservation law for the  solution of \eqref{DNLS}. That is the reason why we consider the Schwartz solution. In addition,  mass threshold is another interesting problem.  Please refer to \cite{HaOz-94, MTX:DNLS:Exist}\cite{Wu-DNLS} and reference therein. Recently, it is very interesting that Bahouri and Perelman obtain global well-posedness for \eqref{DNLS} in $H^{1/2}(\R)$ without mass restriction  by combining the profile decomposition techniques with the integrability structure  in \cite{BahouriPerelman:DNLS:GWP}.
		\item  The above microscopic conservation law  corresponds to macroscopic conservation law of \eqref{DNLS} in \cite{KlausS:DNLS}. In fact, we have
		\begin{align*}A(\vk) =  &  \sgn(\iu\vk^2)\sum_{j=1}^\infty\frac{(-1)^{j-1}}j \tr\left\{\left(\sqrt{R_0}\left(L - L_0\right)\sqrt{R_0}\right)^j\right\} \\
		= & \int_{\R} \rho(\vk) \dx.
		\end{align*}
		 Compared to macroscopic conservation law,  microscopic conservation law with coercivity can be used to show the local smoothing estimate for the solution of \eqref{DNLS} with small mass in $H^s\cap \Schwartz$, $s\in (0,1/2)$ and others. Please refer to  \cite{KV:KdV:AnnMath, KVZ:KdV:GAFA}.
		\item The leading order term in $\rho$ has the following form
		$$\left|\Re \int_{\R}\rho^{[2]}(\vk) dx \right| \approx |\vk|^2  \|q\|^2_{H^{-1/2}_{\vk}},$$
	which captures  $L^2$ norm of the part of $q$ living at frequencies $|\xi|\lesssim |\kappa|^2$ (See also \eqref{LogarithmicBound} and \eqref{A-def'}). Combining  the above conservation law and the similar argument on Besov norm estimate  in \cite[Section ~$3$]{KVZ:KdV:GAFA, KlausS:DNLS} for any Hamiltonian flow  preserving $A(\vk)$ for all $|\vk|\geq 1$, we can obtain  a uniform bound of  $H^s(\R)$ norm of the Schwartz solution of \eqref{DNLS} with small mass 
		\begin{equation}\label{APBound}
		\|q(t)\|_{H^s}\lesssim \|q(0)\|_{H^s}, \text{~~for~~} s\in (0,1/2).
		\end{equation}
		\end{enumerate}
	\end{rem}

 Lastly, the paper is organized as follows. In Section \ref{S:2}, we recall some notations and preliminary estimates. In Section \ref{S:3}, we show the existence and some properties of the Green's function related to the Lax operator $L(\kappa)$. In Section \ref{S:4}, we introduce  the invairant quantity  $A(\kappa)$ from the logarithmic perturbation determinant and show its microscopic conservation laws for the $A(\kappa)$ flow  and  the DNLS flow.

\subsection*{Acknowledgements}
  X. Tang  was supported by NSFC (No. 12001284), and G. Xu  was supported by NSFC (No. 11671046,  and No. 11831004) and by National Key Research and Development Program of China (No. 2020YFA0712900). The authors would like to thank Professor Monica Visan for her valuable comments and suggestions.

\section{Some notation and preliminary estimates}\label{S:2}

In this paper, we take $ r = -\bar{q}$ and  choose
$
s\in (0, \tfrac12), $
and all implicit constants can depend on $s$.  We denote   $\Schwartz$ the Schwartz function,  and introduce the notation
\begin{equation}\label{Bdelta}
\Bd := \left\{q\in H^s:\|q\|_{H^s}\leq \delta\right\}.
\end{equation}
which can be ensured by  scaling argument under the assumption that the mass $\|q\|_{L^2}$ is small enough.

We use the inner product on $L^2(\R)$ as follows
$$
\langle f, g\rangle = \int \overline{f(x)} g(x)\,dx,
$$
 which also gives the dual product between  $H^s(\R)$ and $H^{-s}(\R)$. In addtion, If $F:\Schwartz\to\C$ is $C^1$, we have
\begin{equation}\label{FunctDeriv}
\tfrac{d\ }{d\theta}\Big|_{\theta=0} F(q+\theta f)  = \bigl\langle \bar f, \tfrac{\delta F}{\delta q}\bigr\rangle - \bigl\langle f, \tfrac{\delta F}{\delta r}\bigr\rangle.
\end{equation}
The Fourier transform is defined by
\begin{align*}
\hat f(\xi) = \tfrac{1}{\sqrt{2\pi}} \int_\R e^{-i\xi x} f(x)\,dx,  \qtq{whence}  \widehat{fg}(\xi) = \tfrac{1}{\sqrt{2\pi}} [\hat f * \hat g] (\xi).
\end{align*}

\subsection{Sobolev spaces}

For complex $\kappa$ with  $|\kappa| \geq 1$ and \(\sigma\in \R\) we define the norm
\[
\|q\|_{H^{\sigma}_\kappa}^2 := \int_{\R} \left(4|\kappa|^4 + \xi^2\right)^\sigma |\hat q(\xi)|^2\,d\xi
\]
and write \(H^\sigma  := H^{\sigma}_1\).

For $0<s<\frac12$, simple calculation yields the Sobolev inequality
\begin{align}\label{Linfty bdd}
\|f\|_{L^\infty} \lesssim \|\hat{f}\|_{L^1}\leq \|f\|_{H^{s+\frac{1}2 }_\kappa} \bigl\|(|\xi|^2+4|\kappa|^4)^{-\frac{2s+1}4}\bigr\|_{L^2}\lesssim |\kappa|^{-2s}\|f\|_{H^{s+\frac{1}2 } _\kappa}.
\end{align}
Consequently, we have the following algebra property of $H^{s+\frac{1}2 }_\kappa$ space:
\begin{equation}\label{E:algebra}
\| f g \|_{H^{s+\frac{1}2 }_\kappa} \lesssim |\kappa|^{-2s} \| f \|_{H^{s+\frac{1}2}_\kappa} \| g \|_{H^{s+\frac{1}2}_\kappa}.
\end{equation}

By duality and  the fractional product rule in \cite{ChristW:FracRule}, Sobolev embedding, and \eqref{Linfty bdd}, we obtain
\begin{align}\label{multiplier bdd on ss}
\|qf\|_{H^{s-1/2}}\lesssim 
&  |\kappa|^{-2s} \|q\|_{H^{s-1/2}}\|f\|_{H^{s+\frac12}_\kappa}.
\end{align}

\subsection{Operator estimates and Trace}
For \(0<\sigma<1\) and \(i \kappa^2 \in \R\),  \(|\kappa|\geq 1\) we define the operator \((i \kappa^2 \mp \p)^{-\sigma}\) using the Fourier multiplier \((i \kappa^2 \mp i\xi)^{-\sigma}\) where, for \(\arg z\in (-\pi,\pi]\), we define
\begin{equation}\label{z to the sigma}
z^{-\sigma} = |z|^{-\sigma}e^{-i\sigma \arg z}.
\end{equation}
Therefore for all $\iu \kappa^2 \in \R$ with \(|\kappa|\geq 1\) we have
\[
\left((i\kappa^2 \mp \p)^{-\sigma}\right)^* = (i \kappa^2 \pm \p)^{-\sigma},
\]
and
\[
\left\|(i\kappa^2 \mp \p)^{-\sigma}\right\|_{\op}\leq |\kappa|^{-2\sigma}.
\]

We denote $\I_p$  the Schatten class of compact operators on $L^2(\R)$ whose singular
numbers are $l^p$ summable. $\I_p$ is complete and is an embedded subalgebra of bounded operators on $L^2(\R)$. Moreover, we have 
\begin{equation}\label{Ip: embedding}
\I_p \subset \I_q, \text{~~for~~} p\leq q.
\end{equation}
  Let us recall some facts about the class $\I_p$ that we will use repeatedly in Section \ref{S:3}: An operator A on $L^2(\R)$ is Hilbert–Schmidt class ($\I_2$) if and only if it admits an integral kernel $a(x, y) \in L^2(\R \times \R)$, and
$$\|A\|^2_{\op}\leq \|A\|_{\I_2}^2=\iint_{\R\times \R} |a(x,y)|^2 \, \dx \dy. $$
The product of two Hilbert–Schmidt operators is trace class,; Moreover,  we have
\begin{align*}
\tr(AB):= & \iint_{\R\times \R} a(x,y)b(y,x) \dy \dx = \tr(BA), \\ |\tr(AB)|&  \leq   \|AB\|_{\I_1} \leq \|A\|_{\I_2}\|B\|_{\I_2}
\end{align*}
The class $\I_p$ forms a two-sided ideal in the algebra of
bounded operators on $L^2$; indeed, for any bounded operators $B, C $ on $L^2(\R)$, we have
$$\|BAC\|_{ \I_p} \leq \|B\|_{\op} \|A\|_{\I_p} \|C\|_{\op}.$$
We can refer to  \cite{GohGoldK:book, Simon:Trace} for more details.

The following estimates are the elementary estimates in this paper.

\begin{lem}\label{L:BasicBounds}
Let $s\in (0, 1/2)$, $\iu \kappa^2 \in \R$, we have
\begin{align}
\| (\p + \iu \kappa^2 )^{-\frac12}\kappa q(\p- \iu \kappa^2 )^{ - \frac12}\|_{\I_2}
&  \lesssim  \|q\|_{L^2},    \label{LogarithmicBound}\\ 
\| (\p + \iu \kappa^2 )^{s-\frac12}\kappa q(\p- \iu \kappa^2 )^{ - \frac12}\|_{\I_2} &\lesssim |\kappa| \|q\|_{H^{s-\frac12}_\kappa},  \label{BasicBound}
\end{align}
and 
\begin{align}
\|(\p+ i \kappa^2 )^{-s-\frac12} f (\p- i \kappa^2 )^{-s-\frac12} \|_{\op} &\lesssim |\kappa|^{-2s} \|f\|_{H^{-s-\frac12 }_\kappa}. \label{BasicOpBound}
\end{align}
\end{lem}

\begin{proof} The estimate \eqref{LogarithmicBound} can be proved  by \cite[Lemma~4.1]{KVZ:KdV:GAFA}.   For \eqref{BasicBound},   it suffices to consider the case \(\iu \kappa^2 = 1\) by scaling argument. By Plancherel's Theorem, we have
\begin{align*}
\|(1 - \p)^{s-\frac12} q(1 + \p)^{-\frac12}\|_{\I_2}^2 &=  \tr\left\{(1 - \p^2)^{-\alpha}q(1 - \p^2)^{-\beta}\bar q\right\}\\
&=  \tfrac{1}{2\pi} \iint_{\R^2} \frac{|\hat q(\xi-\eta)|^2\,}{\left(1 + \xi ^2\right)^{\frac12-s}\left(1 + \eta^2\right)^{\frac12}}d\eta\, d\xi,
\end{align*}
Note that
\begin{equation*}
\int_{\R} \frac1{\left(1 + (\xi + \eta)^2\right)^{\frac12 -s}\left(1 + \eta^2\right)^{\frac12}}\,d\eta \lesssim (4 + \xi^2)^{s-\frac12},
\end{equation*}
we can obtain  the estimate \eqref{BasicBound}. 

Lastly, we can obtain \eqref{BasicOpBound}  by duality and \eqref{E:algebra} as follows 
\begin{equation*}
\biggl| \int_{\R} f g h \, dx \biggr|   \lesssim \| f \|_{H^{-s-\frac12}_\kappa} \| gh \|_{H^{s+\frac12}_\kappa} \lesssim |\kappa|^{- 2s}\| f \|_{H^{-s-\frac12}_\kappa} \| g \|_{H^{s+\frac12}_\kappa}  \| h \|_{H^{s+\frac12}_\kappa}.
\end{equation*}
This completes the proof.
\end{proof}

\section{The diagonal Green's functions}\label{S:3}
In this Section, motivated by the ideas in \cite{HKV:NLS, KV:KdV:AnnMath}, we introduce three key quantities $g_{12}$, $g_{21}$, and $\gamma$ from the diagonal Green's function related to  the Lax operator $L(\kappa)$ for \eqref{DNLS}, and establish some elementary estimates about them.  Recall that
\begin{equation}\label{Lax L'}
L(\kappa) = L_0(\kappa) +  \begin{bmatrix}0 & -\kappa q\\-\kappa r&0 \end{bmatrix} \text{~~where~~} L_0(\kappa) := \begin{bmatrix}\p+i\kappa^2 & 0\\0&\p-i\kappa^2\end{bmatrix}.
\end{equation}
Since we  only consider the case $\iu\kappa^2\in\R$ with $|\kappa|\geq 1$, we have
\begin{align}\label{L conjugation}
L(\kappa)^* =   \begin{bmatrix} -\p-\iu \bar{k}^2 & -\bar{\kappa} \bar{r} \\-\bar{\kappa}\bar{q} & -\p+i\bar{\kappa}^2 \end{bmatrix}  
=  -  \begin{bmatrix}1  & 0 \\
0 & -1 \end{bmatrix} 
L(-\bar{\kappa})  \begin{bmatrix} 1 & 0 \\
0 & -1 \end{bmatrix}. 
\end{align}

We now construct the Green's function associated to $L_0(\kappa)$ and $L(\kappa)$, respectively. By the Fourier transformation, the resolvent operator
\begin{equation*}
R_0(\kappa) := L_0(\kappa)^{-1} = \begin{bmatrix}(\p+\iu\kappa^2)^{-1} & 0\\0&(\p-\iu\kappa^2)^{-1}\end{bmatrix}
\end{equation*}
admits the integral kernel
\begin{equation}\label{G_0}
G_0(x,y;\kappa) = e^{-\iu\kappa^2|x - y|}\begin{bmatrix}\bbo_{\{y<x\}}&0\\0&-\bbo_{\{x<y\}}\end{bmatrix} \quad\text{for $\iu\kappa^2 >0$}.
\end{equation}
For $\iu\kappa^2<0$, we may use ${G}_0(x,y;\kappa)=-G_0(y,x;-\bar{\kappa})$ by \eqref{L conjugation}.

By the perturbation theory and the resolvent identity, the resolvent operator $R(\kappa):=L(\kappa)^{-1}$ can be formally  expressed as
\begin{align}
R = & R_0 + \sum_{\ell = 1}^\infty (-1)^{\ell}\sqrt{R_0}\left(\sqrt{R_0}(L - L_0)\sqrt{R_0}\right)^\ell\sqrt{R_0},
 \label{Resolvent} 
\end{align}
where 
\begin{align}
\sqrt{R_0}(L - L_0)\sqrt{R_0} = & -\begin{bmatrix} 0 &\Lambda \\ \Gamma & 0\end{bmatrix}, \label{LambdaGammaJob} \\
\Lambda := (\p+\iu\kappa^2)^{-\tfrac12} \kappa q(\p-\iu\kappa^2)^{-\tfrac12}
\text{~and~} & 
\Gamma :=  (\p-\iu\kappa^2)^{-\tfrac12} \kappa r( \p+\iu\kappa^2)^{-\tfrac12}, \label{D:LambdaGamma}
\end{align}
and fractional powers of  $R_0$ are defined via \eqref{z to the sigma}. By \eqref{LogarithmicBound}, we have
\begin{equation}\label{Lambda}
\|\Lambda\|_{\I_2} = \|\Gamma\|_{\I_2} \lesssim \|q\|_{L^2} \lesssim \delta.
\end{equation}

Now we have the convergence of  series \eqref{Resolvent} as part of the following result.

\begin{prop}[Existence of the Green's function]\label{P:R}
There exists \(\delta>0\) such that $L(\kappa)$ is invertible as an operator on $L^2(\R)$, for all \(q\in \Bd\) and all $\iu \kappa^2 \in \R$ with \(|\kappa|\geq 1\).  The resolvent operator $R(\kappa):=L(\kappa)^{-1}$ admits an integral kernel $G(x,y;\kappa)$ satisfying
\begin{equation}\label{G conjugation}
    \begin{bmatrix}
      G_{11}(x,y,\kappa) & G_{12}(x,y,\kappa) \\
      G_{21}(x,y,\kappa) & G_{22}(x,y,\kappa) \\
    \end{bmatrix}
    =
    -
    \begin{bmatrix}
      \bar{G}_{11}(y,x,-\bar{\kappa}) & \bar{G}_{21}(y,x,-\bar{\kappa}) \\
      \bar{G}_{12}(y,x,-\bar{\kappa}) & \bar{G}_{22}(y,x,-\bar{\kappa}) \\
    \end{bmatrix}
\end{equation}
 such that the mapping
\begin{equation}\label{tensor eqn}
 H^s_\kappa(\R)\ni q \mapsto G - G_0 \in  H^{ s+ \frac12}_\kappa\otimes H^{  s + \frac12 }_\kappa
\end{equation}
is continuous. Moreover, $G-G_0$ is continuous as a function of $(x,y)\in\R\times \R$. Lastly, we have
\begin{align}
\p_xG(x,y;\kappa) &= \begin{bmatrix}-\iu\kappa^2&\kappa q(x)\\ \kappa r(x)& \iu\kappa^2\end{bmatrix}G(x,y;\kappa) + \begin{bmatrix}\delta(x-y)&0\\0&\delta(x-y)\end{bmatrix},\label{xID}\\
\p_yG(x,y;\kappa) &= G(x,y;\kappa)\begin{bmatrix}\iu\kappa^2&- \kappa q(y)\\ -\kappa r(y)&-\iu\kappa^2\end{bmatrix} - \begin{bmatrix}\delta(x-y)&0\\0&\delta(x-y)\end{bmatrix}\label{yID}
\end{align}
in the sense of distributions.
\end{prop}

\begin{proof} The proof is similar as those in \cite[ Proposition $3.1$]{HKV:NLS}. We sketch the proof for completeness.  From \eqref{Lambda}, we have\footnote{Since the convergence of the tails of the series \eqref{Resolvent} is key to the existence of the Green's function, we can pay some regularity $s>s_0>0$ with $s_0\in (0, 1/2)$,  use the $\I_p$ estimate instead of the $\I_2$ estimate to remove  small mass assumption.  Here we pay attention to the regularity problem of the solution for \eqref{DNLS} in this paper. }
$$
\bigl\| \sqrt{R_0}(L - L_0)\sqrt{R_0} \,\bigr\|_{\I_2}\leq \sqrt2 \|\Lambda\|_{\I_2} \lesssim  \|q\|_{L^2} \lesssim \delta
$$
uniformly for $|\kappa|\geq 1$.  Thus, for $\delta>0$ sufficiently small, the series \eqref{Resolvent} converges in operator norm uniformly for $|\kappa|\geq 1$.  It is easy to  verify that the sum acts as an inverse to $L(\kappa)$.

From \eqref{Lambda},  we also have  $R-R_0 \in \I_2$.  In particular, the operator $R-R_0$ admits an integral kernel  $G-G_0$  in $L^2(\R\times \R)$. Moreover, by  \eqref{BasicBound}, we know that $R-R_0$ converges in the sense of Hilbert--Schmidt operators from $H^{ -s -\frac12  }_\kappa$ to $H^{s+\frac 12 }_\kappa$, which implies  \eqref{tensor eqn}.

By \eqref{tensor eqn} , we obtain that the kernel function $G-G_0$  is  continuous in $(x,y)$ since $s+ \frac 12 >\frac12$.

For regular $q$, the identities \eqref{xID} and \eqref{yID} precisely express the fact that $G$ is an integral kernel for $R(\kappa)$. They also hold for irregular $q$ by \eqref{tensor eqn}.
\end{proof}

Let us define  $\vr$, $g_{12}$ and $g_{21}$  as follows:
\begin{align}
\vr(x;\kappa) &:=\sgn(\iu\kappa^2) \bigl[ G_{11}(x,x;\kappa) - G_{22}(x,x;\kappa) \bigr]  - 1, \label{def gama}\\
g_{12}(x;\kappa) &:=  \sgn(\iu\kappa^2) G_{12}(x,x;\kappa), \label{def g12}\\
g_{21}(x;\kappa) &:=  \sgn(\iu\kappa^2) G_{21}(x,x;\kappa), \label{def g21}
\end{align}
where $G_{ij}(x,y,\kappa)$, $1\leq i, j\leq 2$, are the entries of the integral kernel $G(x,y,\kappa)$. If $q\in \BdS$, we may use  \eqref{xID} and \eqref{yID}  to obtain
\begin{align}
\vr' &= 2\kappa\left(qg_{21} - rg_{12}\right),\label{rho-ID}\\
g_{12}' &= -2i\kappa^2 g_{12} -\kappa q[\vr + 1],\label{g12-ID}\\
g_{21}' &=  2i\kappa^2 g_{21} +\kappa r[\vr + 1].\label{g21-ID}
\end{align}
By \eqref{G conjugation}, we have 
\begin{align}
\vr(\kappa) &= \bar \vr(-\bar{\kappa}) \qtq{and} g_{12}(\kappa) = - \bar g_{21}(-\bar{\kappa}).\label{grho-symmetries}
\end{align}
Moreover, by \eqref{rho-ID}, \eqref{g12-ID}, and \eqref{g21-ID} , we have the following identity
\begin{align}
 & \frac{\kappa^2-\vk^2}{\kappa^2}[ g_{12}'(\kappa)g_{21}(\vk)+g_{21}'(\kappa)g_{12}(\vk) ]
\notag \\
&=
[ g_{12}(\kappa)g_{21}(\vk)+g_{21}(\kappa)g_{12}(\vk) ]'
+
\frac{\vk}{2\kappa}[ (\vr(\kappa)+1)(\vr(\vk)+1) ]',\label{micro As commute}
\end{align}
which is closely connected to the commutativity of the $A(\kappa)$'s flows under the Poisson bracket \eqref{PoissonBracket} .

From the series representation \eqref{Resolvent} of   $R(\kappa)$ in $q$ and $r$, we can deduce the corresponding series representations of $g_{12}$, $g_{21}$, and $\gamma$ in $q$ and $r$.  We  use the square brackets notation as follows
\begin{align}\label{g12-Series terms}
g_{12}\sbrack{2m+1}(\kappa) &:= \sgn(\iu\kappa^2)\Big\< \delta_x,\ (\p+i\kappa^2)^{-\frac12}\Lambda \left(\Gamma\Lambda\right)^m(\p-i\kappa^2)^{-\frac12}\delta_x\Big\>, \\
    \label{g21-Series terms}
g_{21}\sbrack{2m+1}(\kappa) &:=  \sgn(\iu\kappa^2)\Big\< \delta_x,\ (\p-i\kappa^2)^{-\frac12}\Gamma \left(\Lambda\Gamma\right)^m(\p+i\kappa^2)^{-\frac12}\delta_x\Big\>,
\end{align}
with $g_{12}\sbrack{2m}(\kappa)=g_{21}\sbrack{2m}(\kappa):= 0$, and similarly, $\gamma\sbrack{2m+1}(\kappa):=0$ and
\begin{align}
 \vr\sbrack{2m}(\kappa) :=&  \sgn(\iu\kappa^2)\Big\< \delta_x,\ (\p+i\kappa^2)^{-\frac12}(\Lambda\Gamma)^m(\p+i\kappa^2)^{-\frac12}\delta_x\Big\> \notag \\
& 
- \sgn(\iu\kappa^2)\Big\< \delta_x,\  (\p-i\kappa^2)^{-\frac12}(\Gamma\Lambda)^m (\p-i\kappa^2)^{-\frac12}\delta_x\Big\>. \label{rho-Series terms}
\end{align}
then by  \eqref{Resolvent}, we have
\begin{equation}\label{g12-gamma-Series}
g_{12}(\kappa) = \sum_{\ell=1}^\infty g_{12}\sbrack{\ell}(\kappa),\quad g_{21}(\kappa) = \sum_{\ell=1}^\infty g_{21}\sbrack{\ell}(\kappa), \text{~and~} \gamma(\kappa) = \sum_{\ell=2}^\infty \gamma\sbrack{\ell}(\kappa) .
\end{equation}
We also write the tails of these series as
\begin{align*}
g_{12}\sbrack{\geq m}(\kappa) := &  \sum_{\ell=m}^\infty g_{12}\sbrack{\ell}(\kappa), \quad g_{21}\sbrack{\geq m}(\kappa) := \sum_{\ell=m}^\infty g_{21}\sbrack{\ell}(\kappa), \\
&  \gamma \sbrack{\geq m}(\kappa) := \sum_{\ell=m}^\infty \gamma \sbrack{\ell}(\kappa).
\end{align*}

By  \eqref{g12-ID} \eqref{g21-ID} and \eqref{QuadraticID},  we obtain  the identities
$$
g_{12} = - (2i \kappa^2+ \p)^{-1} [\kappa q+\kappa \vr q],\quad g_{21} = - (2i \kappa^2-\p)^{-1} [\kappa r+\kappa \vr r],
$$
and
$$
\gamma = -2g_{12} g_{21} -\tfrac12 \gamma^2,
$$
from which we have the explicit expressions for the leading order terms
\begin{gather}
g_{12}\sbrack{1}(\kappa) = - \tfrac {\kappa q}{2i \kappa^2 + \p},\qquad g_{12}\sbrack{3} (\kappa)= \tfrac2{2i\kappa^2 + \p}\big(\kappa q\cdot\tfrac {\kappa r}{2i\kappa^2 -  \p}\cdot \tfrac {\kappa q}{2i \kappa^2 + \p}\big), \label{g12 1 3}\\
g_{21}\sbrack{1}(\kappa) = - \tfrac {\kappa r}{2i \kappa^2 - \p},\qquad\hphantom{-} g_{21}\sbrack{3} (\kappa)= \tfrac{2}{2i\kappa^2 - \p}\big(\kappa r\cdot \tfrac {\kappa q}{2i\kappa^2 + \p}\cdot\tfrac {\kappa r}{2i\kappa^2 - \p}\big),\label{g21 1 3}
\end{gather}
and
\begin{align}
\vr\sbrack{2}(\kappa) &= - 2\,\tfrac {\kappa q}{2i\kappa^2 + \p}\cdot\tfrac{\kappa r} {2i\kappa^2 - \p}, \label{gamma 2}\\
\vr\sbrack{4}(\kappa) &=   \tfrac {\kappa q}{2i\kappa^2 +   \p}\cdot \tfrac 4{2i\kappa^2-  \p}\big(\kappa r\cdot\tfrac {\kappa q}{2i\kappa^2 + \p}\cdot \tfrac {\kappa r}{2i\kappa^2 - \p}\big) \notag \\
&\quad 
        + \tfrac4{2i\kappa^2 + \p}\big(\kappa q\cdot\tfrac {\kappa r}{2i \kappa^2 - \p}\cdot \tfrac {\kappa q}{2i\kappa^2 + \p}\big)\cdot \tfrac{\kappa r} {2i\kappa^2 - \p} \notag\\
&\quad  - 2\,\tfrac {\kappa q}{2i\kappa^2 + \p}\cdot\tfrac {\kappa r}{2i\kappa^2 - \p}\cdot \tfrac {\kappa q}{2i\kappa^2 + \p}\cdot\tfrac {\kappa r}{2i\kappa^2 - \p}. \label{gamma 4}
\end{align}

We   are now ready to obtain some  basic estimates on $g_{12}$, $g_{21}$. 

\begin{prop}[Properties of \(g_{12}\) and \(g_{21}\)]\label{prop:g}
There exists \(\delta>0\) such that for all $\iu \kappa^2\in \R$ with \(|\kappa|\geq 1\) the maps \(q\mapsto g_{12}(\kappa)\) and \(q\mapsto g_{21}(\kappa)\) are (real analytic) diffeomorphisms of \(\Bd\) into \(H^{s+\frac{1}2}\) satisfying the estimates
\begin{equation}\label{g12-Hs}
\|g_{12}(\kappa)\|_{H^{s+\frac{1}2}_\kappa} + \|g_{21}(\kappa)\|_{H^{s+\frac{1}2}_\kappa} \lesssim |\kappa|\, \|q\|_{H^{s-\frac12 }_\kappa}.
\end{equation}
Further, the remainders satisfy the estimate
\begin{equation}\label{g12-LO}
\|g_{12}\sbrack{\geq 3}(\kappa)\|_{H^{s+\frac{1}2}_\kappa} + \|g_{21}\sbrack{\geq 3}(\kappa)\|_{H^{s+\frac{1}2}_\kappa} \lesssim |\kappa|\|q\|_{H^{s-\frac{1}2}_\kappa}\|q\|^2_{L^2},
\end{equation}
uniformly in $\kappa$. Finally, if \(q\) is Schwartz then so are \(g_{12}(\kappa)\) and \(g_{21}(\kappa)\).
\end{prop}
\begin{proof}
It suffices to consider the case \(\iu \kappa^2\geq 1\) as the case \(\iu \kappa^2\leq -1\) is similar, and  by \eqref{grho-symmetries}, it suffices to consider $g_{12}(\kappa)$. Recalling \eqref{g12 1 3}, we obtain
\begin{equation}\label{linear isomorphism}
\|g_{12}\sbrack{1}(\kappa)\|_{H^{s+\frac{1}2}_\kappa} = |\kappa|\, \|q\|_{H^{s-\frac 12 }_\kappa}.
\end{equation}
To bound the remainder terms in the series\footnote{If we use $\I_p$ with $p>2$ instead of $\I_2$ once again, we may pay some regularity on $s>s_0>0$ to remove the small mass assumption.}, we employ duality and Lemma~\ref{L:BasicBounds}:
\begin{align}\label{trilinear est}
\bigl|\<f, g_{12}\sbrack{\geq 3}(\kappa)\>\bigr|
&\leq \|(\p-i\kappa^2 )^{-s-\frac{1}2}\bar f(\p + i\kappa^2 )^{-s-\frac{1}2}\|_{\op} \notag\\
&\qquad\qquad \times\sum_{\ell=1}^\infty\|(\p+i\kappa^2 )^{s-\frac12}kq(\p-i\kappa^2 )^{-\frac12}\|_{\I_2}^{2\ell+1}|\kappa|^{-2s(2\ell-1)} \notag\\
&\lesssim |\kappa|^{-2s} \|f\|_{H^{-s-\frac{1}2}_\kappa}\sum_{\ell=1}^\infty \left(|\kappa| \|q\|_{H_\kappa^{s-\frac{1}2}}  \right)^{2\ell+1} |\kappa|^{-2s(2\ell-1)} \notag\\
&\lesssim\|f\|_{H^{-s-\frac{1}2}_\kappa} |\kappa|\, \|q\|_{H_\kappa^{s-\frac{1}2}}\sum_{\ell=1}^\infty \left(|\kappa|^{1-2s} \|q\|_{H_\kappa^{s-\frac{1}2}}  \right)^{2\ell}  \notag\\
&\lesssim\|f\|_{H^{-s-\frac{1}2}_\kappa} |\kappa|\, \|q\|_{H_\kappa^{s-\frac{1}2}} \|q\|^{2}_{L^2}  
\end{align}
provided $\delta>0$ is sufficiently small.  This proves \eqref{g12-LO} and completes the proof of \eqref{g12-Hs}.

In addtion, we  have
\begin{equation*}
\tfrac{\delta g_{12}}{\delta q}(\kappa)\bigr|_{q=0} = - \frac{\kappa}{ 2i\kappa^2 + \p} \qtq{and}  \tfrac{\delta g_{12}}{\delta r}(\kappa)\bigr|_{q=0} = 0
\end{equation*}
which is an isomorphism, as noted already in \eqref{linear isomorphism}.  Furthermore, for any  $f\in \Schwartz$, we have
\begin{align*}
\left.\frac d{d\epsilon}\right|_{\epsilon = 0} G(x,z;q + \epsilon f) &= -\int G(x,y;q)\begin{bmatrix}0&f(y)\\ - \bar f(y)&0\end{bmatrix}G(y,z;q)\,dy.
\end{align*}
By similar analysis as those to prove \eqref{g12-LO}, we have
\begin{align*}
& \bigl\| \tfrac{\delta g_{12}}{\delta r}(\kappa) \bigr\|_{H^{s-\frac 12}_\kappa\rightarrow H^{s+\frac{1}2}_\kappa} + \bigl\| \tfrac{\delta g_{12}}{\delta q}(\kappa)+ \frac{\kappa}{ 2i\kappa^2 + \p}  \bigr\|_{H^{s-\frac 12}_\kappa\rightarrow H^{s+\frac{1}2}_\kappa} 
\lesssim   \|q\|^2_{L^2} 
  \lesssim       \delta^2,
\end{align*}
and so the inverse function theorem implies the diffeomorphism property for sufficiently small $\delta$ .

The regulairty result can be easily obtained in a similar argument as those in  \cite[Proposition $2.2$]{KV:KdV:AnnMath} and \cite[Proposition~3.2]{HKV:NLS}.
\end{proof}

We  also have some  estimates on $\gamma$ as follows.
\begin{prop}[Properties of \(\vr\)]\label{prop:rho}
There exists \(\delta>0\) such that for all $\iu \kappa^2\in \R$ with  \(|\kappa|\geq 1\) the map \(q\mapsto \vr(\kappa)\) is bounded from \(\Bd\) to \(L^1\cap H^{s+\frac{1}2}\) and we have the estimates
\begin{align}
\|\vr(\kappa)\|_{H^{s+\frac{1}2}_\kappa} &\lesssim |\kappa|^{2-2s}\|q\|_{H^{s-\frac{1}2}_\kappa}^2,\label{rho-Hs}\\
\|\vr(\kappa)\|_{L^\infty} &\lesssim |\kappa|^{2-4s}\|q\|_{H^{s-\frac 12}_\kappa}^2,\label{rho-Linfty}\\
\|\vr(\kappa)\|_{L^1} &\lesssim |\kappa |^{2} \|q\|_{H^{-1}_\kappa}^2 + |\kappa|^{-2(4s-1)}\|q\|_{H^{s-\frac 12}_\kappa}^4,\label{rho-L1}\\
\|\vr\sbrack{\geq 4}(\kappa)\|_{L^1}&\lesssim |\kappa|^{-2(4s-1) }\|q\|_{H^{s-\frac 12}_\kappa}^4,\label{rho-LO}
\end{align}
uniformly in \(\kappa\).
Further, we have the quadratic identity
\begin{equation}
\label{QuadraticID}
 \vr + \frac12 \vr^2 = -2 g_{12}g_{21},
\end{equation}
and if \(q\) is Schwartz, then so is \(\vr(\kappa)\).
\end{prop}

\begin{proof}
Once again it suffices to consider the case \(\iu \kappa^2\geq1\). Using \eqref{gamma 2} and \eqref{E:algebra}, we obtain
\begin{align*}
\|\gamma\sbrack{2}\|_{H_\kappa^{s+\frac12}}\lesssim \|\tfrac {\kappa q}{2i\kappa^2 + \p}\cdot\tfrac{\kappa r} {2i\kappa^2 - \p} \|_{H_\kappa^{s+\frac12}}\lesssim |\kappa|^{2-2s}\| q\|_{H^{s-\frac12}_\kappa}^2.
\end{align*}
To handle $\gamma\sbrack{\geq 4}$ we use the series representation \eqref{g12-gamma-Series} and the same dual argument used to prove \eqref{g12-LO}. The estimate \eqref{rho-Linfty} then follows from \eqref{rho-Hs} via~\eqref{Linfty bdd}.

Choosing $\vk=\kappa$ in \eqref{micro As commute}, we obtain
$$
\partial_x \bigl\{ 2 g_{12}(x;\kappa)g_{21}(x;\kappa) + \tfrac12\vr(x;\kappa)^2 + \vr(x;\kappa) \bigr\} = 0.
$$
By \eqref{g12-Hs} and \eqref{rho-Hs}, we know that the term in the braces vanishes as $|x|\to\infty$. Thus the identity \eqref{QuadraticID} follows by integration.

By using this quadratic identity, we may write
\begin{equation}\label{E:vr gr 4}
\vr\sbrack{\geq 4} = - \tfrac12 \vr^2 - 2g_{12}\sbrack{\geq 3}\cdot g_{21} - 2g_{12}\sbrack{1}\cdot g_{21}\sbrack{\geq 3}.
\end{equation}
By Proposition~\ref{prop:g} and  \eqref{rho-Hs}, we have
\begin{align*}
\|g_{12}\sbrack{\geq 3}\|_{L^2}+\|g_{21}\sbrack{\geq 3}\|_{L^2}&\lesssim |\kappa|^{-2(s+\frac12)}\bigl[\|g_{12}\sbrack{\geq 3}\|_{H^{s+\frac12}_\kappa}+ \|g_{21}\sbrack{\geq 3}\|_{H^{s+\frac12}_\kappa}\bigr] \lesssim |\kappa|^{-6s+2}\|q\|_{H^{s-\frac12 }_\kappa}^3,\\
\|g_{12}\|_{L^2}+\|g_{21}\|_{L^2}&\lesssim |\kappa|^{-2(s+\frac12)} \bigl( \|g_{12}\|_{H_\kappa^{s+\frac12}}+\|g_{21}\|_{H_\kappa^{s+\frac12}}\bigr)\lesssim   |\kappa|^{-2s}\|q\|_{H^{s-\frac12}_\kappa}, 
\end{align*} and
\begin{align*}
\|\vr\|_{L^2}&\lesssim |\kappa|^{-2(s+\frac12)}\|\vr\|_{H^{s+\frac12}_\kappa}\lesssim|\kappa|^{-4s+1}\|q\|_{H^{s-\frac12}_\kappa}^2,
\end{align*}
which together with H\"older's inequality imply that
\begin{align*}
\|\vr\sbrack{\geq 4}\|_{L^1} &\lesssim \|\vr\|_{L^2}^2 + \|g_{12}\sbrack{\geq 3}\|_{L^2}\|g_{21}\|_{L^2} + \|g_{12}\sbrack{1}\|_{L^2}\|g_{21}\sbrack{\geq 3}\|_{L^2}\lesssim |\kappa|^{-2(4s-1)}\|q\|_{H^{s-\frac12}_\kappa}^4,
\end{align*}
thus we obtain \eqref{rho-LO}. The estimate \eqref{rho-L1} then follows from applying the Cauchy-Schwarz inequality to \eqref{gamma 2}. 

The regulairty result can also be deduced by a similar argument as those in \cite[Proposition $2.2$]{KV:KdV:AnnMath} and \cite[Proposition~3.3]{HKV:NLS}.
\end{proof}

Due to the structure of microscopic conservation law in \eqref{micro A},  the combination function \(\frac{g_{12}(\kappa)}{2 + \vr(\kappa)}\) will be also used later. We now give the analogue estimates. Firstly, we denote:
\begin{equation*}
\tfrac{g_{12}}{2 + \vr} = \big(\tfrac{g_{12}}{2 + \vr}\big)\sbrack{1} + \big(\tfrac{g_{12}}{2 + \vr}\big)\sbrack{ 3} + \big(\tfrac{g_{12}}{2 + \vr}\big)\sbrack{\geq 5},
\end{equation*}
where the leading order terms are given by
\begin{equation}\label{more sbrack}
\big(\tfrac{g_{12}}{2 + \vr}\big)\sbrack{1} = \tfrac12 g_{12}\sbrack{1} \qtq{and} \big(\tfrac{g_{12}}{2 + \vr}\big)\sbrack{3} = \tfrac12 g_{12}\sbrack{3} - \tfrac14g_{12}\sbrack{1}\vr\sbrack{2},
\end{equation}
and the remainder term is given by
\begin{align}
\big(\tfrac{g_{12}}{2 + \vr}\big)\sbrack{\geq 3} &= \tfrac12 g_{12}\sbrack{\geq 3} - \tfrac{g_{12} \vr}{2(2 + \vr)}. \label{more sbrack'}
\end{align}

We can now show the following estimates about $\tfrac{g_{12}(\kappa)}{2 + \vr(\kappa)}$.

\begin{cor}\label{C:ET}
Let $s\in (0,1/2)$ and $q\in B_\delta$.  there exists \(\delta>0\) such that for all $\iu \kappa^2\in \R$ with  $|\kappa|\geq 1$, we have the estimates
\begin{align}
|\kappa|^2\bigl\|\tfrac{g_{12}(\kappa)}{2 + \vr(\kappa)}\bigr\|_{H^{s-1/2}} + \bigl\| \tfrac{g_{12}(\kappa)}{2 + \vr(\kappa)}\bigr\|_{H^{s+1/2}}
	&\lesssim |\kappa| \|q\|_{H^{s-1/2}},\label{ET Sob} \\
|\kappa|^2\bigl\|\big(\tfrac{g_{12}(\kappa)}{2 + \vr(\kappa)}\big)\sbrack{\geq 3}\bigr\|_{H^{s-1/2}} + \bigl\| \big(\tfrac{g_{12}(\kappa)}{2 + \vr(\kappa)}\big)\sbrack{\geq 3}\bigr\|_{ H^{s+1/2}}
    &\lesssim   |\kappa| \|q\|_{H^{s-1/2}}\|q\|_{L^2}^2. \label{ET1 Sob}
\end{align}
\end{cor}

\begin{proof}
From \eqref{more sbrack} and \eqref{g12 1 3}, we see that
$$
|\kappa|^2\bigl\|\big(\tfrac{g_{12}}{2 + \vr}\big)\sbrack{1}\bigr\|_{H^{s-1/2}} + \bigl\| \big(\tfrac{g_{12}}{2 + \vr}\big)\sbrack{1}\bigr\|_{ H^{s+1/2}}
\approx \bigl\| (2i\kappa^2+\p) \big(\tfrac{g_{12}(\vk)}{2 + \vr(\vk)}\big)\sbrack{1}\bigr\|_{ H^s}
\approx|\kappa| \|q\|_{H^{s-1/2}}.
$$
Thus  it suffice for \eqref{ET Sob} to show \eqref{ET1 Sob}.  Moreover,  by \eqref{g12-ID}, we have
$$
(2i\kappa^2+\p) \big(\tfrac{g_{12}}{2 + \vr}\big)\sbrack{\geq 3} = - \kappa \tfrac{\vr}{2(2+\vr)} q + \tfrac{g_{12}}{(2 + \vr)^2} \gamma', 
$$
and therefore, we have 
\begin{align*}
\text{LHS of \eqref{ET1 Sob}} &\lesssim |\kappa| \bigl\| \tfrac{\vr}{2+\vr} q \bigr\|_{H^{s-1/2}}  + \bigl\| \tfrac{g_{12}}{(2 + \vr)^2} \gamma' \bigr\|_{H^{s-1/2}} \\
&\lesssim |\kappa|^{1-2s} \|q\|_{H^{s-1/2}} \bigl\| \tfrac{\vr}{2+\vr} \bigr\|_{ H^{s+1/2}_\kappa}
		+ |\kappa|^{-2s} \|\vr'\|_{H^{s-1/2}}  \bigl\| \tfrac{g_{12}}{(2 + \vr)^2} \bigr\|_{ H^{s+1/2}_\kappa}\\
&\lesssim |\kappa|^{1-2s} \|q\|_{H^{s-1/2}} \bigl\|\vr \bigr\|_{ H^{s+1/2}_\kappa}		\\
&\lesssim |\kappa| \|q\|_{H^{s-1/2}}  \|q\|^2_{L^2},
\end{align*}
where the second step we use \eqref{multiplier bdd on ss} and \eqref{rho-Hs}, and the third step we  expand as series and employ the algebra property \eqref{E:algebra}, together with \eqref{g12-Hs} and \eqref{rho-Hs}.  This yields \eqref{ET Sob} for sufficiently small $\delta$.
\end{proof}

\section{Conservation laws and dynamics}\label{S:4}
In this section, we will firstly introduce the invariant quantity $A(\kappa)$ from the logrithmic perturbation determinant, which is related to the integrability and spectral invariance of \eqref{DNLS}, then we show the dynamics and microscopic conservation laws  of the $A(\kappa)$'s flow and the DNLS (i.e. $H_{\DNLS}$) flow, respectively. 

\subsection{Conservation Laws}
Inspired by \cite{HKV:NLS, KV:KdV:AnnMath,KVZ:KdV:GAFA, Rybkin:KdV:Cons Law}, we formally define the logarithmic perturbation determinant $\sgn(\iu\kappa^2)\log\det(L_0^{-1}L)$   as follows:
\begin{equation}\label{A-def}
A(\kappa; q,r) := \sgn(\iu\kappa^2)\sum_{j=1}^\infty\frac{(-1)^{j-1}}j \tr\left\{\left(\sqrt{R_0}\left(L - L_0\right)\sqrt{R_0}\right)^j\right\}.
\end{equation}
 By \eqref{LambdaGammaJob}, simple calculations deduce  that for $q,r\in \Schwartz$, we have
\begin{equation}\label{A-def'}
A(\kappa; q,r) = - \sgn(\iu\kappa^2)\sum_{m=1}^\infty \tfrac{1}{m} \tr\left\{(\Lambda\Gamma)^m \right\}.
\end{equation}
In the following, we will use \eqref{A-def'} as the definition of $A(\kappa; q,r)$.
For the sake of simplicity, we  write
\begin{equation}\label{Am to A}
	A(\kappa; q,r) = \sum\limits_{m=1}^{\infty} A_m(\kappa; q,r),  \; A_m(\kappa; q,r) :=-\sgn(\iu\kappa^2)\tfrac{1}{m} \tr\left\{(\Lambda\Gamma)^m \right\}.
\end{equation}

Firstly,  we have
\begin{lem}[Properties of $A$]\label{L:A}
	There exists $\delta>0$ such that for all $q\in \Bd \cap \Schwartz$ and $i \kappa^2\in \R$ with $|\kappa|\geq 1$, the series \eqref{A-def'}  converges absolutely. Moreover, we have
	\begin{gather}
	\tfrac{\delta\,}{\delta q} A(\kappa) =-\kappa g_{21}(\kappa),
	\quad \tfrac{\delta\,}{\delta r}A(\kappa) = -\kappa g_{12}(\kappa), \label{gij from A}\\
	\; \vr'(\kappa) = 2\left(-q\tfrac{\delta\,}{\delta q}A(\kappa) + r\tfrac{\delta\,}{\delta r}A(\kappa) \right).
	\label{rho from A}
	\end{gather}
\end{lem}
\begin{proof}
	By \eqref{Lambda}, we know that  $	A(\kappa; q,r)$  converges absolutely  for all $q\in \Bd\cap \Schwartz$.
	By simple calculations, we have
	\begin{equation}\label{rho from Am}
	\tfrac{\delta\,\,}{\delta q}A_m  =-\kappa g_{21}\sbrack{2m-1}
	\qtq{and}
	\tfrac{\delta\,\,}{\delta r}A_m =-\kappa g_{21}\sbrack{2m-1},
	\end{equation}
	which together with \eqref{rho-ID},  \eqref{g12-gamma-Series} implies \eqref{rho from A}.
\end{proof}
Next, we show that the mass, the Hamiltonian and the energy defined by  \eqref{HDNLS} and \eqref{Mass Energy} respectively, arise as the coefficients in the asymptotic expansion of $A(\kappa)$ as $|\kappa| \to\infty$. More precisely, we have.
\begin{lem}[Asymptotic expansion of $A(\kappa)$]\label{L:A asymptotics} For $q\in B_\delta\cap \Schwartz$, we have, as $|\kappa|\to\infty$,
	\begin{equation}\label{A asymptotics}
	A(\kappa)=(-\iu)\tfrac{M}{2}  - \tfrac{(-\iu)^2}{2i\kappa^2}\frac{H_{\DNLS}}{2}   +\frac{(-i)^3}{(2i\kappa^2)^2} \frac{E_{\DNLS}}{2}+ O(|\kappa|^{-6}). 
	\end{equation}
\end{lem}
\begin{proof}
	By \eqref{g12-ID} and \eqref{g21-ID}, and  \eqref{rho from A}, we have
	\begin{equation*}
	\begin{split}
 2 \iu \kappa ^2 \tfrac{\delta A}{\delta q}
	= & \; 
 \partial \tfrac{\delta A}{\delta q}  + 2 \kappa^2 r \cdot \partial^{-1}\left(-q \tfrac{\delta A}{\delta q} + r \tfrac{\delta A}{\delta r}\right) + \kappa^2 r ,
	\\
 2 \iu \kappa ^2 \tfrac{\delta A}{\delta r}
= & -
\partial \tfrac{\delta A}{\delta r}  + 2 \kappa^2 q \cdot \partial^{-1}\left(-q \tfrac{\delta A}{\delta q} + r \tfrac{\delta A}{\delta r}\right) + \kappa^2 q.
	\end{split}
	\end{equation*}
	By the asymptotic analysis, we have
	\begin{equation}\label{biHam2}
	\begin{split}
	\tfrac{\delta A}{\delta q} &=  -\kappa g_{21} \\
		& =-\iu
	\frac{1}{2}r+\frac{1}{2\iu\kappa^2}\bigl(-\frac{ \iu}{2} r'+\frac{1}{2} qr^2 \bigr)+ \frac{1}{(2\iu \kappa^2)^2} \left(-\frac{\iu}{2}r'' + \frac{3}{2} qr r' + \frac{3}{4} \iu q^2r^3\right) + O(|\kappa|^{-6}),
	\\
	\tfrac{\delta A}{\delta r} &=  -\kappa g_{12} \\
& =-\iu
\frac{1}{2}q+\frac{1}{2\iu\kappa^2}\bigl(\frac{ \iu}{2} q'+\frac{1}{2} q^2r \bigr)+ \frac{1}{(2\iu \kappa^2)^2} \left(-\frac{\iu}{2}q'' - \frac{3}{2} qq'r  + \frac{3}{4} \iu q^3r^2\right) + O(|\kappa|^{-6}),
	\end{split}
	\end{equation}
which together with the fact that
	\begin{equation}\label{recovA}
	A(q,r) = \int_0^1 \partial_\theta A(\theta q,\theta r)\pd\theta
	\end{equation}
imply \eqref{A asymptotics}.
\end{proof}
\begin{rem}
	By computing $\vr= 2 \partial^{-1}\left( - q\tfrac{\delta A}{\delta q} + r\tfrac{\delta A}{\delta r}\right)$,
	we can obtain asymptotic expansion for $\vr$,
	\begin{align}
	\label{vr asymptotics}
	\vr   =
	\frac{1}{2\kappa^2}qr-\frac{1}{4\kappa^4}\bigl( \iu q'r-\iu qr'+\frac{3}{2}q^2r^2 \bigr)+O(|\kappa|^{-6}).
	\end{align}
\end{rem}
\begin{lem}[Density function of $A(\kappa)$]\label{L:micro A}
	For all $q\in \Bd\cap\Schwartz$ and $i \kappa^2\in \R$ with $|\kappa|\geq 1$, we have
	\begin{gather}
	A(\kappa) = - \bar A(-\bar{\kappa}),    \label{A-Symmetries} \\
	\tfrac{\partial\, }{\partial \kappa}A(\kappa) =
	\int_{\R}\Bigl[ 2i\kappa \cdot \vr(\kappa)- \bigl(qg_{21}(\kappa)+rg_{12}(\kappa)\bigr)\Bigr]\dx ,
	\label{rho-to-A} \\
	A (\kappa) = \int_{\R} \rho(\kappa)\dx, \qtq{where} \rho(\kappa)=-\kappa\frac{qg_{21}(\kappa)+rg_{12}(\kappa)}{2+\vr(\kappa)}.  \label{micro A}
	\end{gather}
\end{lem}
\begin{proof}
	Firstly, we show that \eqref{rho-to-A}. In fact, simple calculations imply that
	\begin{equation}\label{rho to Am}
	\frac{\partial\ }{\partial \kappa}A_m
	=
	\int_{\R}\Bigl[ 2\iu\kappa\gamma\sbrack{2m} - \bigl(q g_{21}\sbrack{2m-1}+ r g_{12}\sbrack{2m-1} \bigr)\Bigr]\dx.
	\end{equation}
	By summation with respect to $m$, we can obtain \eqref{rho-to-A}.
	
	We now estimate \eqref{micro A}.
	On one hand, 
	by differentiating \eqref{g12-ID}, \eqref{g21-ID}, and \eqref{QuadraticID} with respect to $\kappa$,
	we obtain that,
	\begin{align*}
	\partial_x\Bigl(g_{12}\tfrac{\partial g_{21}}{\partial\kappa} - \tfrac{\partial g_{12}}{\partial\kappa} g_{21}\Bigr)
	=
	&
	\kappa(qg_{21}+rg_{12})\tfrac{\partial\ }{\partial\kappa}(\vr+1)
	-\kappa(\vr+1)\tfrac{\partial\ }{\partial\kappa}(qg_{21}+rg_{12})
	\\
	&
	-2\iu\kappa\vr(\vr+2)
	+(\vr+1)(qg_{21}+rg_{12}).
	\end{align*}
	Using \eqref{rho-ID}, we have
	\begin{align*}
	- \bigl( g_{12}\tfrac{\partial g_{21}}{\partial\kappa} - \tfrac{\partial g_{12}}{\partial\kappa} g_{21}\bigr) \vr'  =  -\vr(2+\vr)\tfrac{\partial\ }{\partial\kappa}\bigl(qg_{21}-rg_{12}\bigr)+ \bigl(qg_{21}-rg_{12}\bigr)(1+\vr)\tfrac{\partial\vr}{\partial\kappa}.
	\end{align*}
	Combining the above two identities, we get
	\begin{align*}
	\partial_x\frac{g_{12}\tfrac{\partial g_{21}}{\partial\kappa} - \tfrac{\partial g_{12}}{\partial\kappa} g_{21}}{2+\vr} &= \bigl(qg_{21}+rg_{12}-2\iu\kappa\vr\bigr) - \frac{\partial\ }{\partial\kappa}\bigl(\kappa\frac{qg_{21}+rg_{12}}{2+\vr}\bigr) ,
	\end{align*}
	which can be integrated in $x$ to yield
	\begin{align}\label{iden dAdk}
	\frac{\partial\ }{\partial\kappa} \int\kappa \frac{qg_{21}+rg_{12}}{2+\vr}\dx = \int qg_{21}+rg_{12}-2\iu\kappa\vr\dx = -\frac{\partial A}{\partial\kappa} .
	\end{align}
	
	 On the other hand, by \eqref{A asymptotics}, we have
	 \begin{equation}\label{A expand}
	 A(\kappa)=-i\tfrac{1}{2} M+O( |\kappa|^{-2} ), \qtq{as} \kappa\to\infty.
	 \end{equation} and by \eqref{biHam2} and \eqref{vr asymptotics}, we obtain
	\begin{equation}\label{vr expand}
	\rho(\kappa)=-\iu\frac{1}{2}qr + O(|\kappa|^{-2}), \qtq{as} \kappa\to\infty.
	\end{equation}
Combining \eqref{iden dAdk}, \eqref{A expand} and \eqref{vr expand}, we can obtain	\eqref{micro A}.
	
	Lastly,  \eqref{A-Symmetries} is obvious from \eqref{grho-symmetries} and \eqref{micro A}.
\end{proof}
Next, we show the commutation of $A(\kappa)$'s under the Poisson bracket \eqref{PoissonBracket},  which implies that $A(\vk)$ is an invariant quantity under the $A(\kappa)$'s flows.
\begin{lem}[Poisson brackets]\label{lem:PoissonBrackets}
	There exists $\delta>0$ such that for any $\kappa$ and $\vk$ with $\iu\kappa^2, \, \iu \vk^2 \in \R\setminus(-1,1)$ with $\kappa^2\neq\vk^2$ and any $q\in \BdS$ we have
	\begin{equation}\label{APoissonBracket}
	\{A(\kappa),A(\vk)\} = 0.
	\end{equation}
\end{lem}
\begin{proof}
	By \eqref{rho from A}  and \eqref{PoissonBracket}, we have
	\begin{align*}
	\{A(\kappa),A(\vk)\} &=\kappa\vk \int g_{12}(\kappa)g_{21}'(\vk) + g_{21}(\kappa)g_{12}'(\vk)\dx,
	\end{align*}
	which together with \eqref{micro As commute} implies that \eqref{APoissonBracket} holds. .
\end{proof}

Next we will exhibit the dynamics of $g_{12}$,  $g_{21}$ and  $\gamma$ defined by \eqref{def gama}, \eqref{def g12} and \eqref{def g21}  along  the $A(\kappa)$ flow and DNLS  flow respectively.
\subsection{Dynamics  I:  the $A(\kappa)$ flow.}
Firstly, we have
\begin{lem}[Dynamics of $A(\kappa)$ flow]\label{L:A flow}
	Let $i\kappa^2\in\R$ with $|\kappa|>1$.  Under the $A(\kappa)$ flow, we have
	\begin{align}\label{qr under A}
	\frac{\pd\ }{\pd t}q = - \kappa g_{12}'(\kappa) \qtq{and} \frac{\pd\ }{\pd t}r =-  \kappa g_{21}'(\kappa).
	\end{align}
\end{lem}
\begin{proof}
	It is obvious from  \eqref{rho from A}  and \eqref{HFlow}.
	\end{proof}

\begin{lem}[Lax pair for the $A(\kappa)$ flow]\label{P:Lax A} 
	Let $\iu\kappa^2, \iu\vk^2\in\R\setminus(-1,1)$   with $\kappa^2\neq\vk^2$, and $L(\vk)$ be defined by \eqref{Intro KN L}. Under the $A(\kappa)$ flow, we have 
	\begin{align}\label{Lax rep A flow}
	\frac{\pd\ }{\pd t} L(\vk) = [ P_{A(\kappa)}, L(\vk)],
	\end{align}
	where
	\begin{equation}\label{A B}
P_{A(\kappa)}=
	\begin{bmatrix}-\frac 12 \Xi (\vr(\kappa) + 1) & - \Theta g_{12}(\kappa)\\-\Theta g_{21}(\kappa)&\frac 12 \Xi (\vr(\kappa) + 1)\end{bmatrix}	\text{~~with~~}
	\Theta=\tfrac{\vk \kappa^3}{\kappa^2-\vk^2}, ~~ \Xi=\tfrac{\vk^2\kappa^2}{\kappa^2-\vk^2}.
	\end{equation}
\end{lem}
\begin{proof}
Firstly, by \eqref{rho-ID}  and the fact that
\begin{align*} \vr'(\kappa) 
= &  (\partial+\iu\vk^2)\bigl( \vr(\kappa)+1 \bigr)-\bigl( \vr(\kappa)+1 \bigr)(\partial+\iu\vk^2) \\
= &  (\partial-\iu\vk^2)\bigl( \vr(\kappa)+1 \bigr)-\bigl( \vr(\kappa)+1 \bigr)(\partial-\iu\vk^2), 
	 \end{align*}
 we have
	\begin{align}
	\label{L11}
 (\partial+\iu\vk^2)\bigl( \vr(\kappa)+1 \bigr)-\bigl( \vr(\kappa)+1 \bigr)(\partial+\iu\vk^2)-2\kappa\bigl(qg_{21}(\kappa) - rg_{12}(\kappa)\bigr)=0.
	\\
	\label{L22}
(\partial-\iu\vk^2)\bigl( \vr(\kappa)+1 \bigr)-\bigl( \vr(\kappa)+1 \bigr)(\partial-\iu\vk^2)-2\kappa\bigl(qg_{21}(\kappa) - rg_{12}(\kappa)\bigr)=0.
	\end{align}

Next, by \eqref{g12-ID} and the fact that
	\begin{equation*}
	g_{12}'(\kappa)+2\iu\vk^2g_{12}(\kappa)
	=
	(\partial+\iu\vk^2)g_{12}(\kappa) -g_{12}(\kappa)(\partial-\iu\vk^2),
	\end{equation*}
 we get
	\begin{align}\label{L12}
	(\kappa^2-  \vk^2 )g_{12}'(\kappa)
	=
	\kappa^2\bigl[ (\partial+\iu\vk^2)g_{12}(\kappa) -g_{12}(\kappa) (\partial-\iu\vk^2)\bigr]+\kappa\vk^2 q\bigl(\vr(\kappa)+1\bigr).
	\end{align}
	
Finally, by   \eqref{g21-ID} and the fact that
	\begin{equation*}
	g_{21}'(\kappa)-2\iu\vk^2g_{21}(\kappa)
	=
	(\partial-\iu\vk^2)g_{21}(\kappa) -g_{21}(\kappa)(\partial+\iu\vk^2),
	\end{equation*}
we obtain
	\begin{align}\label{L21}
	(\kappa^2 - \vk^2 )g_{21}'(\kappa)
	=
	\kappa^2\bigl[ (\partial-\iu\vk^2)g_{21}(\kappa)  - g_{21}(\kappa)(\partial+\iu\vk^2) \bigr]-\kappa\vk^2 r\bigl(\vr(\kappa)+1\bigr).
	\end{align}
	
Combining \eqref{Intro KN L},  \eqref{qr under A} with
	\eqref{L11}, \eqref{L22}, \eqref{L12},  \eqref{L21},  we can obtain \eqref{Lax rep A flow}. 
\end{proof}

\begin{prop}[Microscopic conservation law for the $A(\kappa)$ flow]\label{Thm:A flow}
	Let $\iu\vk^2, \iu \kappa^2\in\R\setminus(-1,1)$ with $\kappa^2 \neq\vk^2$. Under the $A(\kappa)$ flow, we have  
	\begin{align}
	\frac{\pd\ }{\pd t} g_{12}(\vk) &=-\Xi g_{12}(\vk) \left[\vr(\kappa)+1\right] + \Theta g_{12}(\kappa)\left[\vr(\vk)+1\right] , \label{Ady g12}\\
	\frac{\pd\ }{\pd t} g_{21}(\vk) &= \Xi g_{21}(\vk) \left[\vr(\kappa)+1\right] - \Theta g_{21}(\kappa)\left[\vr(\vk)+1\right] , \label{Ady g21}\\
	\frac{\pd\ }{\pd t} \vr(\vk) &=-
	2 \Theta \left[ g_{12}(\kappa) g_{21}(\vk) -  g_{21}(\kappa)g_{12}(\vk) \right],\label{Ady gam}
	\end{align}
and the following microscopic conservation laws 
	\begin{align}
	\label{Ady gam claw}
	\partial_t\bigl\{ 2i\vk\cdot \vr(\vk) -\left(q g_{21}(\vk)+r g_{12}(\vk)\right)\bigr\}
	+ &  \partial_x j_{\vr}(\vk,\kappa)=0,\\
	\label{A flow micr cons law}
	\partial_t \rho(\vk) +  \partial_x j_{A(\kappa)}(\vk,\kappa)  =0, & 
	\end{align}
	where the flux functions $	j_{\vr}$ and $j_{A(\kappa)}$ are determined by
	\begin{align}\label{A:rho dot}
	j_{\vr}(\vk,\kappa):
	= &  -
	\tfrac{ \kappa^3(\kappa^2+\vk^2) }{ ( \kappa^2-\vk^2 )^2 }
	\bigl[ g_{12}(\kappa)g_{21}(\vk)+g_{21}(\kappa)g_{12}(\vk)\bigr] \notag \\& 
-
	\tfrac{\kappa^4\vk}{ ( \kappa^2-\vk^2 )^2 }
	\bigl[( \vr(\kappa)+1 )( \vr(\vk)+1 ) \bigr],
\\ \label{j sub A}
	j_{A(\kappa)}(\vk,\kappa):= & -\Theta \tfrac{ g_{12}(\kappa)g_{21}(\vk)+g_{12}(\vk)g_{21}(\kappa) }{2 +  \vr(\vk) }
	-\Theta \frac{\vk}{2\kappa}
	\vr(\kappa)
	.
	\end{align}
\end{prop}

\begin{rem}Due to the coercivity of the quadratic term of the flux, the conservation law \eqref{A flow micr cons law} is more useful than \eqref{Ady gam claw}.
	\end{rem}
\begin{proof}
	Firstly, we show the dynamics \eqref{Ady g12}, \eqref{Ady g21} and \eqref{Ady gam}. By  \eqref{HFlow}, Proposition \ref{P:R} and Lemma \ref{P:Lax A}, we have
	\begin{align*}
	\frac{\pd\ }{\pd t} G(x,z;\vk)
	&=
	- \int G(x,y;\vk)\frac{\pd\ }{\pd t} L(\vk)G(y,z;\vk)\pd{y} \\
	&= \int G(x,y;\vk)
	\begin{bmatrix}
	0 & -\kappa^2g_{12}'(\kappa) \\
	-\kappa^2g_{21}'(\kappa) & 0
	\end{bmatrix}G(y,z;\vk)\pd{y}
	\\
	&=
	P_{A(\kappa)}(x;\kappa,\vk)G(x,z;\vk)-G(x,z;\vk)P_{A(\kappa)}(z;\kappa,\vk).
	\end{align*}
By choosing $x=z$ and \eqref{Lax rep A flow}, we have
	\begin{align*}
	\frac{\pd\ }{\pd t} G(x,x;\vk)
	= & 
-	\Xi
	\begin{bmatrix}
	0
	&g_{12}(\vk)\bigl( \vr(\kappa)+1 \bigr)
	\\
	-g_{21}(\vk)\bigl( \vr(\kappa)+1 \bigr)
	& 0
	\end{bmatrix}
	\\ & 
	-
	\Theta
	\begin{bmatrix}
	g_{12}(\kappa)g_{21}(\vk)-g_{21}(\kappa)g_{12}(\vk)
	& -g_{12}(\kappa)\bigl( \vr(\vk)+1 \bigr)
	\\
	g_{21}(\kappa)\bigl( \vr(\vk)+1 \bigr)
	& -g_{12}(\kappa)g_{21}(\vk)+g_{21}(\kappa)g_{12}(\vk)
	\end{bmatrix},
	\end{align*}
	which implies  \eqref{Ady g12}, \eqref{Ady g21} and \eqref{Ady gam}.
	
	Next, we prove  \eqref{Ady gam claw}. A direct computation implies that
	\begin{align}
	\notag & 
	\partial_t\bigl\{2i\vk\vr(\vk)- q g_{21}(\vk)-r g_{12}(\vk)\bigr\}
\\
	=&
	\label{A:rho dot 1}
	- g_{21}(\vk)\frac{\pd\ }{\pd t}q
	- g_{12}(\vk)\frac{\pd\ }{\pd t}r
	+ 2\iu\vk\frac{\pd\ }{\pd t}\vr(\vk)
	- q \frac{\pd\ }{\pd t}g_{21}(\vk)
	- r \frac{\pd\ }{\pd t}g_{12}(\vk).
	\end{align}
	On the one hand, by Lemma \ref{L:A flow},  we have
	\begin{equation}\label{A:rho dot 11}
	- g_{21}(\vk)\frac{\pd\ }{\pd t}q
- g_{12}(\vk)\frac{\pd\ }{\pd t}r  
	=
	\kappa\left[ g_{21}(\vk)g_{12}'(\kappa) + g_{12}(\vk)g_{21}'(\kappa) \right].
	\end{equation}
	On the other hand, by \eqref{Ady g12},  \eqref{Ady g21},  and \eqref{Ady gam}, we obtain that
	\begin{align}
	\notag
	&  2\iu\vk\frac{\pd\ }{\pd t}\vr(\vk)
	- q \frac{\pd\ }{\pd t}g_{21}(\vk)
	- r \frac{\pd\ }{\pd t}g_{12}(\vk) 
	\\
	=
	\notag
	&
	-4 i\vk\Theta g_{12}(\kappa)g_{21}(\vk)
	-\Xi g_{21}(\vk) q\left[ \vr(\kappa)+1 \right]
	+\Theta g_{21}(\kappa) q\left[ \vr(\vk)+1 \right]
	\\
	\notag
	&+4 i\vk\Theta g_{12}(\kappa)g_{21}(\vk)
	+\Xi g_{12}(\vk) r\left[ \vr(\kappa)+1 \right]
	-\Theta g_{12}(\kappa) r\left[ \vr(\vk)+1 \right]
	\\
	=
	&
	\frac{\vk\Theta}{\kappa^2}\left[ g_{21}(\vk)g_{12}'(\kappa) + g_{12}(\vk)g_{21}'(\kappa) \right]
	-\frac{\Theta}{\vk}\left[ g_{21}(\kappa)g_{12}'(\vk) + g_{12}(\kappa)g_{21}'(\vk) \right].
	\label{A:rho dot 21}
	\end{align}
Combining \eqref{A:rho dot 1}, \eqref{A:rho dot 11},  \eqref{A:rho dot 21} with \eqref{micro As commute}, we can obtain
	\eqref{A:rho dot}. 
	
	Finally, let us show that \eqref{j sub A} holds. A direct calculation implies that
	\begin{align} 
	& \; \bigl( \vr(\vk)+2 \bigr)^2\partial_t \rho(\vk)
\notag	\\
	=
	&
	\label{j sub A 1}
	-\vk\bigl[
	g_{21}(\vk)\frac{\pd\ }{\pd t}q+ g_{12}(\vk)\frac{\pd\ }{\pd t}r
	\bigr]\bigl( \vr(\vk)+2 \bigr)
	\\
	\label{j sub A 2}
	&
	-
	\vk\bigl[
	q \frac{\pd\ }{\pd t}g_{21}(\vk)+r \frac{\pd\ }{\pd t}g_{12}(\vk)
	\bigr]\bigl( \vr(\vk)+2 \bigr)
	\\
	\label{j sub A 3}
	&
	+\vk\bigl[
	q g_{21}(\vk)+r g_{12}(\vk)
	\bigr]\partial_t \vr(\vk).
	\end{align}
	By Lemma \ref{L:A flow} and \eqref{Ady g12}, \eqref{Ady g21}, we have
	\begin{align}\label{j sub A 12}
	\notag
	&  \eqref{j sub A 1}+\eqref{j sub A 2}
	\\
	\notag
	= &
	\vk\kappa\left[ g_{21}(\vk)g_{12}'(\kappa)+g_{12}(\vk)g_{21}'(\kappa) \right]\bigl( \vr(\vk)+2 \bigr)
	\\
	\notag
	&+\frac{\vk\Theta}{2\kappa}\vr'(\kappa)\bigl( \vr(\vk)+1 \bigr)\bigl( \vr(\vk)+2 \bigr)
	-\frac{\Xi}{2}\vr'(\vk)\bigl( \vr(\kappa)+1 \bigr)\bigl( \vr(\vk)+2 \bigr)
	\\
	\notag
	= &
	\vk\kappa\left[ g_{21}(\vk)g_{12}'(\kappa)+g_{12}(\vk)g_{21}'(\kappa) \right]\bigl( \vr(\vk)+2 \bigr)
	\\
	&-\frac{\Xi}{2}\left[\bigl(\vr(\kappa)+1\bigr)\bigl( \vr(\vk)+1 \bigr)\right]'\bigl( \vr(\vk)+2 \bigr)
	+\Xi\vr'(\kappa)\bigl( \vr(\vk)+1 \bigr)\bigl( \vr(\vk)+2 \bigr). 
	\end{align}
	By  \eqref{Ady gam}, we have
	\begin{align}\label{j sub A 3 1}
	\notag
	\eqref{j sub A 3}
	= &
	-2\vk\Theta
	\left[qg_{21}(\vk)+rg_{12}(\vk)\right]
	\left[ g_{12}(\kappa)g_{21}(\vk)-g_{21}(\kappa)g_{12}(\vk) \right]
	\\
	\notag
	=&
	-2\vk\Theta
	\left[qg_{21}(\vk)-rg_{12}(\vk)\right]
	\left[ g_{12}(\kappa)g_{21}(\vk)+g_{21}(\kappa)g_{12}(\vk) \right]
	\\
	\notag
	&
	+4\Theta\vk g_{12}(\vk)g_{21}(\vk)\left[qg_{21}(\kappa)-rg_{12}(\kappa)\right]
	\\
	=&
	-\Theta
	\vr'(\vk)
	\left[ g_{12}(\kappa)g_{21}(\vk)+g_{21}(\kappa)g_{12}(\vk) \right]
	- \frac{\vk\Theta}{2\kappa}\vr'(\kappa) \vr(\vk)\left[ \vr(\vk)+2 \right].
	\end{align}
Combining \eqref{j sub A 12}, \eqref{j sub A 3 1} and  \eqref{micro As commute},  we can deduce
	\eqref{j sub A} . 
\end{proof}
\subsection{Dynamics  II: the DNLS  flow:.} Now we turn to the $H_{\DNLS}$ flow. We firstly recall the  Lax representation for \eqref{DNLS} as follows.
\begin{lem}[Lax Pair for the DNLS flow, \cite{AbCl:book,KaupN:DNLS}]\label{P:Lax H} Let $\iu\vk^2\in\R\setminus(-1,1)$, $L(\vk)$ be defined by \eqref{Intro KN L}. Under the DNLS  flow, we have
	\begin{align}\label{Lax repres}
	\frac{\pd\ }{\pd t} L(\vk) = [ P_{H_{\DNLS}}, L(\vk)],
	\end{align}
	where $
	P_{H_{\DNLS}}=
	\begin{bmatrix} -2i \vk^4-i\vk^2 q r & 2\vk^3q + i\vk q' + \vk q^2 r\\
	2\vk^3 r - i \vk r' + \vk qr^2 &2i \vk^4 + i \vk^2 q r \end{bmatrix}$.
\end{lem}

Following the analogue argument as those in Proposition \ref{Thm:A flow}, we can obtain main result in this paper.
\begin{thrm}[Microscopic conservation law for the DNLS flow]\label{Thm:H flow}
	Let $\iu\vk^2\in\R\setminus(-1,1)$. Under the DNLS flow, we have
	\begin{align}\label{qr under DNLS}
	\frac{\pd\ }{\pd t}q = i q''+(q^2r)' \qtq{and} \frac{\pd\ }{\pd t}r =-  i r''+(qr^2)'.
	\end{align}
and 
	\begin{equation} 
	\begin{aligned}
	\frac{\pd\ }{\pd t} g_{12}(\vk) &=-2\left(2\iu \vk^4 + i\vk^2 q r \right) g_{12}(\vk) - \left(2\vk^3 q + i\vk q' + \vk q^2 r\right) \left(\vr(\vk) + 1\right) , \\
	\frac{\pd\ }{\pd t} g_{21}(\vk) &= 2\left(2\iu \vk^4 + i\vk^2 q r \right) g_{21}(\vk) + \left(2\vk^3 r - i\vk r' + \vk q r^2\right) \left(\vr(\vk) + 1\right) , \\
	\frac{\pd\ }{\pd t} \vr(\vk) &= 2\vk^2\vr'(\vk) + 2i\vk(q'g_{21}(\vk)+r'g_{12}(\vk)) + qr\vr'(\vk).
	\end{aligned}
	\end{equation}
Moreover, we have the following microscopic conservation law
\begin{equation}\label{MCL: H Flow}
\partial_t \rho(\vk) + \partial_x j_{\DNLS}(\vk)=0, 
\end{equation}
	where the density $\rho$ and the flux $j_{\DNLS}$  are defined by \eqref{micro A} and 
	\begin{align}\label{j sub H}
	j_{\DNLS}(\vk):
	= \, &
	i \vk \frac{q'g_{21}(\vk) - r' g_{12}(\vk)}{2+\vr(\vk)}
	- i\vk^2 q r
	- 2\vk^2 \rho(\vk)
	- qr\rho(\vk).
	\end{align}
\end{thrm}
\begin{proof}
	 By  \eqref{HFlow}, Proposition \ref{P:R} and Lemma \ref{P:Lax H},  it is easy to obtain the dynamics of $q$, $r$ and $g_{12}$, $g_{21}$ and $\gamma$ under the $H_{\DNLS}$ flow.  Combining the dynamics of $q$, $r$, $g_{12}$, $g_{21}$, $\gamma$ and \eqref{micro As commute}, we can complete the proof of the microscopic conservation law \eqref{MCL: H Flow}.  
	\end{proof}

\bibliographystyle{plain}

\end{document}